\documentclass[10pt,reqno]{amsart}

\usepackage{amsmath}
\usepackage{amssymb}
\usepackage{graphicx}        
\usepackage{multicol}        
\usepackage[bottom]{footmisc}
\usepackage{bcol}
\usepackage[numbers]{natbib}

\newcommand{\cN}{{N}}
\newcommand{\cNini}{\cN^{-}_i}
\newcommand{\cNouti}{\cN^{+}_i}

\newcommand{\cM}{{ M}}
\newcommand{\cT}{{ T}}
\newcommand{\cK}{{ K}}

\newcommand{\xijk}{{x_{ij}^k}}
\newcommand{\xjik}{{x_{ji}^k}}

\newcommand{\tij}{{t_{ij}}}
\newcommand{\tji}{{t_{ji}}}

\newcommand {\beqn}{\begin{equation}}\newcommand {\eeqn}{\end{equation}}
\newcommand {\beqan}{\begin{eqnarray}}\newcommand {\eeqan}{\end{eqnarray}}
\newcommand {\beqa}{\begin{eqnarray*}}\newcommand {\eeqa}{\end{eqnarray*}}
\newcommand \proj[2]{\textup{proj}_{#1}(#2)}

\newcommand{\R}{\ensuremath{\mathbb{R}}}
\newcommand{\Z}{\ensuremath{\mathbb{Z}}}

\newcommand{\be}[1]{\begin{equation}\label{#1}}
\newcommand{\ee}{\end{equation}}

\makeindex             

\def\FMS{\ensuremath{{F}_{MS}} }

\newcommand{\ZZ}{\Z}

\newcommand{\ceil}[1]{\ensuremath{\lceil#1\rceil}}

\newcommand{\FLOOR}[1]{\ensuremath{\lfloor #1 \rfloor}}
\newcommand{\CEIL}[1]{\ensuremath{\lceil#1\rceil}}
\usepackage[colorinlistoftodos,prependcaption,textsize=small]{todonotes}

\newcommand{\sm}{\setminus}

\newcommand{\UT}{U(T)} 
\newcommand{\UTs}{U(T^*)} 	
\newcommand{\UTe}{U^=(T)} 	
\newcommand{\UTse}{U^=(T^*)} 	
\newcommand{\UTh}{U(\widehat T)}
\newcommand{\DT}{D(T)}
\newcommand{\BT}{B(T)}
\newcommand{\BTs}{B(T^*)} 	
\newcommand{\BdTs}{B(2 T^*)} 	
\newcommand{\BdTse}{B^=(2 T^*)} 	
 	
\newcommand{\BTse}{B^=(T^*)} 	

\usepackage{amsthm}
\newtheorem{proposition}{Proposition}
\newtheorem{lemma}{Lemma}

\newtheorem{corollary}{Corollary}
\theoremstyle{definition}
\newtheorem{definition}{Definition}
\theoremstyle{remark}
\newtheorem{remark}{Remark}

\title{On Capacity Models for Network Design}

\author{Alper Atamt\"urk and Oktay G\"unl\"uk}

\thanks{ \noindent \hskip -2.5em
	Alper Atamt\"urk: Industrial Engineering \& Operations Research, University of
	California, Berkeley, CA 94720 USA. \texttt{atamturk@berkeley.edu} \\
Oktay G\"unl\"uk: Mathematical Sciences, IBM T. J. Watson Research Center, Yorktown
Heights, NY 10598 USA. \texttt{gunluk@us.ibm.com}
}

\begin{document}
	
\maketitle

\begin{abstract}
{
	In network design problems capacity constraints are modeled in three different ways depending on the application and the underlying technology for installing capacity:
	\textit{directed, bidirected,} and \textit{undirected}. In the literature, polyhedral investigations for strengthening mixed-integer formulations are done separately for each model. In this paper, we examine the relationship between these models to provide a unifying approach and show that one can indeed translate valid inequalities from one to the others. In particular, we show that the projections of the undirected and bidirected models onto the capacity variables are the same. We demonstrate that valid inequalities previously given for the undirected and bidirected models can be derived as a consequence of the relationship between these models and the directed model.
}
\end{abstract}

\begin{center}
November 2017
\end{center}

\BCOLReport{17.07}

\section{Introduction}

In network design problems capacity constraints are modeled in three different ways depending on the application and the underlying technology for installing capacity:
\textit{directed, bidirected,} and \textit{undirected}.
In \textit{directed} models, the total flow on an arc is limited by the capacity of the directed arc.
In \textit{bidirected} models, if a certain capacity is installed on an arc, 
then the same capacity also needs to be installed on the reverse arc. Whereas
in \textit{undirected} models, the sum of the flow on an arc and its reverse arc is limited by the capacity of the undirected edge associated with the two arcs.

In the literature polyhedral investigations for strengthening mixed-integer formulations are done separately for each model.
\citet{MM:st-path,MMV:conv-2core,MMV:modeling-2fnlp,M:proj-netload,A:kpart,A:Hamid:2015} consider the undirected capacity model.
\citet{BG:cnd,G:bc-nd} study the bidirected capacity model. Whereas
\citet{MW:nd,BCGT:cap-inst,ANS:avub,A:nd,AR:arc,AG:mixset,AR:mcf-sep} consider the directed capacity model.
In this paper, we examine the relationship between 
these three separately-studied models to provide a unifying approach 
and show how one can translate valid inequalities from one to the others.
In particular, we show that the projections of the undirected and bidirected models onto the capacity variables are the same.
We demonstrate that valid inequalities previously given for undirected and bidirected models can be derived as a consequence of the relationship between these models and the directed model.

Let $G=(N,E)$ be an undirected graph with node set $N$ and edge set $E$. 
Let $A$ be the arc set obtained from $E$ by including the arcs in each direction for the edges in $E$.
Let $\cM$ denote the set of facility types where a  unit of facility $m\in\cM$ provides capacity $c_m$. 
Depending on the model, facilities are installed either on the edges or on the arcs of the network and accordingly, we use $\bar c\in\R^E$ or $\bar c\in\R^A$ to denote the existing capacities on the edges or arcs of the network.
Without loss of generality, we  assume that $c_m\in\Z$ for all $m\in\cM$ and $c_1<c_2<\ldots<c_{|M|}$.
Let the demand data for the problem be given by the square matrix $\cT=\{\tij\}$, where $\tij\ge0$ is the amount of (directed) traffic that 
needs to be routed from node $i\in \cN$ to $j\in \cN$. 
Let $K=\{(ij)\in N\times N\::\:i\not=j\}$ and define the $|K|\times|N|$ demand matrix $D=\{d_u^k\}$,  where
$$
d_u^k = \begin{cases}
\phantom{-}\tij & \text{ if } u=j,\\
-\tij & \text{ if } u=i,\\
\phantom{-}0 & \text{ otherwise, }
\end{cases}
$$
for $(ij)=k\in K$ and $u\in \cN$.

\section{Undirected capacity model}
In the undirected network design problem, the sum of the flow on an arc and its reverse arc is limited by the capacity of the undirected edge associated with the two arcs.
This problem can be formulated as follows:
\beqan \nonumber 
 \text{min~~~~~} ~f(x,y)\hskip2cm&&\\[.3cm]
\text {s.t.~~~~~~~} 
\sum_{j \in \cNouti} \xijk ~~- \sum_{j \in \cNini} \xjik~ &=& d_i^k, ~~~ {\rm for\ } k \in \cK, \  i\in \cN, 
\label{eq:flowxx} \\ \label{eq:capxx}
    \sum_{k \in \cK } \xijk ~~+~ \sum_{k \in \cK } \xjik ~~  &\leq& \bar c_e+ \sum_{m \in \cM } \ c_m  y_{m,e},~~~ {\rm for\ } e=\{i,j\} \in E,\\[.2cm] \label{eq:nonnegxx} x,~y~&\ge& 0,\quad y\in\Z^{M\times E},
\eeqan 
where  $x$ and $y$ denote the flow and capacity variables respectively and $f$ denotes the cost function. 
In most applications the function $f$ can be decomposed as $f(x,y)=f^1(x)+f^2(y)$ and, furthermore, it is typically linear.

Let  $\UT$ denote the set of feasible solutions to inequalities \eqref{eq:flowxx}--\eqref{eq:nonnegxx}.
A capacity vector $\bar y$  {\em accommodates} traffic $T$  if there exists a feasible flow vector $x$ such that $(x,\bar y)\in \UT$.
In other words, $\bar y$  accommodates  $T$ if  $\bar y\in \proj y{\UT}$, where $\proj y{\cdot}$ denotes the orthogonal projection operator onto the space of the $y$ variables.
We say that two traffic matrices $T$ and $\widehat T$ are {\em pairwise similar} if  $\tij+\tji=\hat t_{ij}+\hat t_{ji}$ for all $(ij)\in K$.
We next show that $\proj y{\UT}=\proj y{\UTh}$ provided that $T$ and $\widehat T$ are {pairwise similar}.

\begin{lemma}\label{lem1}
	Capacity vector $y$ accommodates  $T$ if and only if it accommodates all $\widehat T$ {pairwise similar} to  $T$.
\end{lemma}

\ignore{
\begin{proof}  
	It is sufficient to argue that the claim holds when  $T$ and $\widehat T$ differ only in two entries corresponding to traffic between the pair of nodes $u,v\in N$, in such a way that $\hat t_{uv}+\hat t_{vu}= t_{uv}+ t_{vu}$. 
	One can then repeat the argument for the remaining pairs of nodes. 
	
	Let  $u,v\in N$ be such that $\hat t_{uv}+\hat t_{vu}>0$.
	Assuming $y$ can accommodate  $T$, let $x$ be a  flow vector such that $(x,y)\in \UT$.
	Construct  $\hat x$ from $x$ by letting all entries of $\hat x$ corresponding to $k\in K\sm \{uv,vu\}$ be same as that of $x$.
	Let $\alpha = \hat t_{uv}/(\hat t_{uv}+\hat t_{vu})=\hat t_{uv}/( t_{uv}+ t_{vu})$.
	For the remaining entries of $\hat x$ associated with commodities $uv$ and $vu$, we set $\hat x_{ij}^{uv} = \alpha (x_{ij}^{uv} +x_{ji}^{vu} )$ and $\hat x_{ij}^{vu} = (1-\alpha) (x_{ij}^{vu} +x_{ji}^{uv} )$ for all $(i,j)\in A$. 
	
	Notice that $\hat x_{ij}^{vu}+\hat x_{ji}^{vu}+\hat x_{ij}^{uv}+\hat x_{ji}^{uv} =  x_{ij}^{vu}+ x_{ji}^{vu}+ x_{ij}^{uv}+ x_{ji}^{uv} $ for all $(ij)\in A$, and therefore $\hat x$ satisfies the capacity constraints \eqref{eq:capxx}.
	In addition, $\hat x$ also satisfies the flow balance constraints  \eqref{eq:flowxx} as 
	$\hat  d_i^{uv}=\alpha (d_i^{uv}+ d_i^{vu})$ and $\hat  d_i^{vu}=(1-\alpha) (d_i^{uv}+ d_i^{vu}$) for all $i\in N$.
	Therefore, the desired property holds when $T$ and $\widehat T$ differ only in two entries. 
	Repeating the same argument for the remaining pairs of nodes proves the claim.
\end{proof}
}
\begin{proof}  
Let $T$ and $\widehat T$ be pairwise similar. Consider a pair of nodes $u,v\in N$ with nonzero traffic, i.e., $ \sigma = t_{uv}+ t_{vu} = \hat t_{uv}+\hat t_{vu} > 0$. 

Assuming $y$ accommodates  $T$, let $x$ be a  flow vector such that $(x,y)\in \UT$.
Construct  $\hat x$ from $x$ by letting all entries of $\hat x$ corresponding to $k\in K\sm \{uv,vu\}$ be same as that of $x$.
Let $\alpha = \hat t_{uv}/\sigma$.
For the remaining entries of $\hat x$ associated with commodities $uv$ and $vu$, we set $\hat x_{ij}^{uv} = \alpha (x_{ij}^{uv} +x_{ji}^{vu} )$ and $\hat x_{ij}^{vu} = (1-\alpha) (x_{ij}^{vu} +x_{ji}^{uv} )$ for all $(ij)\in A$. 

Notice that $\hat x_{ij}^{vu}+\hat x_{ji}^{vu}+\hat x_{ij}^{uv}+\hat x_{ji}^{uv} =  x_{ij}^{vu}+ x_{ji}^{vu}+ x_{ij}^{uv}+ x_{ji}^{uv} $ for all $\{i,j\}\in E$. Therefore, $\hat x$ satisfies the capacity constraints \eqref{eq:capxx}.
In addition, $\hat x$ also satisfies the flow balance constraints  \eqref{eq:flowxx} as 
$\hat  d_i^{uv}=\alpha (d_i^{uv}+ d_i^{vu})$ and $\hat  d_i^{vu}=(1-\alpha) (d_i^{uv}+ d_i^{vu}$) for all $i\in N$.
Repeating the same argument for the remaining pairs of nodes proves the claim.
\end{proof}

\ignore{
**** OLD DEFN ***
We say  that  a function $f$ is {\em arc-symmetric} if  $f(x,y)=f(\hat x,y)$ whenever $\hat x$ differs from $x$ only in  components $\hat x_{a}^{k}$  for $a\in\{(ij),(ji)\}$ and $k\in\{(uv),(vu)\}$ where $(ij),(ji)\in A$ and ${uv},{vu}\in K$ in such a way that 
$$\hat x_{ij}^{vu}+\hat x_{ji}^{vu}+\hat x_{ij}^{uv}+\hat x_{ji}^{uv} =  x_{ij}^{vu}+ x_{ji}^{vu}+ x_{ij}^{uv}+ x_{ji}^{uv}. $$
***
}

\begin{definition}
An objective function $f$ is called {\em arc-symmetric} if  $f(x,y)=f(\hat x,y)$ whenever
$$x_{ij}^{vu}+ x_{ji}^{vu}+ x_{ij}^{uv}+ x_{ji}^{uv}  = \hat x_{ij}^{vu}+\hat x_{ji}^{vu}+\hat x_{ij}^{uv}+\hat x_{ji}^{uv} 
\text{ for } { uv},{vu}\in K \text{ and } \{i,j\} \in E.  $$
\end{definition}

\begin{lemma}\label{lem2}
Let $T$ and $\widehat T$ be  { pairwise similar} and $f$ be arc-symmetric.
If $(x,y)\in \UT$, then there exists $(\hat x,y)\in \UTh$ such that $f(x,y)=f(\hat x,y)$.
\end{lemma}
\begin{proof}
As in the proof of Lemma~\ref{lem1}, it is possible to construct a flow vector $\hat x$ such that $(\hat x,y)\in \UTh$. Furthermore, as $\hat x_{ij}^{vu}+\hat x_{ji}^{vu}+\hat x_{ij}^{uv}+\hat x_{ji}^{uv} =  x_{ij}^{vu}+ x_{ji}^{vu}+ x_{ij}^{uv}+ x_{ji}^{uv} $ for all $(ij)\in A$ and the  function is arc-symmetric, the result follows.
\end{proof}

Given a traffic matrix  $T$, we define its symmetric counterpart to be $T^*=(T+T^T)/2$.
In other words, $t_{uv}^*=t_{vu}^*= (t_{uv}+ t_{vu})/2$.
Also note that $T$ and $T^*$ are pairwise similar.
We have so far established that optimizing
an arc-symmetric cost function $f$ over $\UT$
is equivalent to optimizing it over $\UTs$.
\begin{lemma}\label{lem3}
	Let  $f$ be an arc-symmetric  function and $T^*$ be a symmetric matrix.
	If $(x,y)\in \UTs$, then there exists $(\hat x,y)\in \UTs$ such that $\hat x_{ij}^{uv}=\hat x_{ji}^{vu}$ for all $(ij)\in A$ and $(uv)\in K$. Furthermore, $f(x,y)=f(\hat x,y)$.
\end{lemma}
\begin{proof}
	Let $\tilde x$ be such that $\tilde x_{ij}^{uv}=x_{ji}^{vu}$ for all $(ij)\in A$ and $(uv)\in K$. 
	As $T^*$ is symmetric, $(\tilde x,y)\in \UTs$.
	Furthermore, by convexity, defining $\hat x = (x+\tilde x)/2$ we have  $(\hat x,y)\in \UTs$.
	In addition, $\hat x_{ij}^{uv}=\hat x_{ji}^{vu}$ for all $(ij)\in A$ and $(uv)\in K$.
	Finally, as $f$ is arc-symmetric and $\hat x_{ij}^{vu}+\hat x_{ji}^{vu}+\hat x_{ij}^{uv}+\hat x_{ji}^{uv} =  x_{ij}^{vu}+ x_{ji}^{vu}+ x_{ij}^{uv}+ x_{ji}^{uv} $, for all $(ij)\in A$ and $(uv)\in K$, the claim holds.
\end{proof}

Let $\UTe$ denote  the set of feasible solutions $(x,y)$ to constraints \eqref{eq:flowxx}-\eqref{eq:nonnegxx} 
together with the following equations
\beqn  x_{ij}^{uv}= x_{ji}^{vu} \text{~~~ for all }(ij)\in A,~(uv)\in K.\label{eqeq}\eeqn
Lemma \ref{lem3} in conjunction with Lemma \ref{lem2} establishes that  when $f$ is an arc-symmetric  function,  optimizing $f$ over $\UT$ is same as optimizing it over $\UTse$.

\begin{proposition}\label{cor1}
Let $f$ be an arc-symmetric function,  $T$ be a traffic matrix and let $T^*$ be its symmetric counterpart. Then
$$\min_{(x,y)\in\UT} f(x,y) ~=~\min_{(x,y)\in\UTs} f(x,y)~=~\min_{(x,y)\in\UTse} f(x,y).$$
Furthermore, given an optimal solution to any one of the problems, optimal solutions to the other two can be constructed easily.
\end{proposition}

Furthermore, notice that if \eqref{eqeq} holds, then 
\beqn \sum_{k \in \cK } \xijk ~~+~ \sum_{k \in \cK } \xjik~ = ~2 \max\Big\{\sum_{k \in \cK } \xijk ,\sum_{k \in \cK } \xjik\Big\}\label{eq:capa} \eeqn
for all $(ij)\in A$.
We next relate these observations on undirected capacity models to network design problems with bidirected capacity constraints.

\section{Bidirected capacity model}
In the bidirected network design problem, the total flow on an arc and total flow on its reverse arc are each limited by the capacity of the undirected edge associated with these arcs.
This problem can be formulated as follows:
\beqan \nonumber 
\text{min~~~~~} ~f(x,y)\hskip2cm&&\\[.3cm]
\text {s.t.~~~~~~~} 
\sum_{j \in \cNouti} \xijk ~~- \sum_{j \in \cNini} \xjik~ &=& d_i^k, ~~~ {\rm for\ } k \in \cK, \  i\in \cN, 
\label{eq:flowxxb} \\ \label{eq:capxxb}
\max\Big\{\sum_{k \in \cK } \xijk ,\sum_{k \in \cK } \xjik\Big\} ~~  &\leq& \bar c_e+ \sum_{m \in \cM } \ c_m  y_{m,e},~~~ {\rm for\ } e=\{i,j\} \in E,\\[.2cm] \label{eq:nonnegxxb} x,~y~&\ge& 0,\quad y\in\Z^{M\times E}.   
\eeqan
Let $\BT$ be the set of feasible solutions $(x,y)$ to inequalities \eqref{eq:capxxb}--\eqref{eq:nonnegxxb}.
We next show that if $T$ is symmetric and $(x,y)\in \BT$, then there exists a flow vector  $\hat x$ such that 
$(\hat x,y)\in \BT$ and $\hat x$ satisfies  \eqref{eqeq}.  
Furthermore$f$ is an arc-symmetric cost function, then $f(x,y)=f(\hat x,y)$. 

\begin{lemma}\label{lem4}
	Let  $f$ be an arc-symmetric  function and $T^*$ be a symmetric matrix.
	If $(x,y)\in \BTs$, then there exists $(\hat x,y)\in \BTs$ such that $\hat x_{ij}^{uv}=\hat x_{ji}^{vu}$ for all $(ij)\in A$ and $(uv)\in K$. Furthermore, $f(x,y)=f(\hat x,y)$.
\end{lemma}
\begin{proof} The proof is essentially identical to that of Lemma \ref{lem3}. 
	First we construct $\tilde x\in \UTs$ by letting $\tilde x_{ij}^{uv}=x_{ji}^{vu}$ for all $(ij)\in A$ and $(uv)\in K$. 
	Then we define  $\hat x = (x+\tilde x)/2$ and observe that  $(\hat x,y)\in \UTs$ and that it satisfies the properties in the claim.
\end{proof}

Therefore, if $T^*$ is a symmetric traffic matrix, then
\beqn\min_{(x,y)\in\BTs} f(x,y) ~=~\min_{(x,y)\in\BTse} f(x,y),\label{eq:BTse}\eeqn
where, $\BTse$ is the  the set of feasible solution $(x,y)$ to constraints \eqref{eq:flowxxb}--\eqref{eq:nonnegxxb} together with equations   \eqref{eqeq}.
We next show that optimizing an arc-symmetric cost function  over $\UT$ is equivalent to optimizing a slightly different function over $\BdTs$. 
\begin{proposition}\label{thm1}
Let $f$ be an arc-symmetric function,  $T$ be a traffic matrix and let $T^*$ be its symmetric counterpart. 
Then
$$\min_{(x,y)\in\UT} f(x,y) ~=~\min_{(x,y)\in\BdTs} g(x,y),$$
where  $g(x,y)= f(\frac12x,y)$.
Furthermore, optimal solutions to each problem can be easily mapped into the other.
\end{proposition}
\begin{proof}
By Corollary \ref{cor1}, we have	$\min_{(x,y)\in\UT} f(x,y) =\min_{(x,y)\in\UTse} f(x,y)$ and a common optimal solution to both problems can easily be constructed from an optimal solution to the first one.
In addition, by Lemma \ref{lem4}, 
$\min_{(x,y)\in\BTs} f(x,y) =\min_{(x,y)\in\BTse} f(x,y)$ and a common optimal solution to both problems can easily be constructed from an optimal solution to the first one.
Therefore, it suffices to show the result for  $\UTse$ and $\BdTse$ instead of $\UTs$ and $\BdTs$.

Consider a point  $p=(\bar x,\bar y)\in\UTse$. As $\bar x$ satisfies equation \eqref{eqeq} and therefore \eqref{eq:capa}, we have $p'=(2 \bar x,\bar y)\in\BdTse$.  
As  $g(x,y)= f(\frac12x,y)$, we also have $g(p')=f(p)$.
Conversely, and given a solution $q=( x', y')\in\BdTse$ we construct  $q'=(\frac12 x', y')\in\UTse$ with $g(q)=f(q')$. Combining these observations, we conclude that 
$$\min_{(x,y)\in\UTe} f(x,y) ~=~\min_{(x,y)\in\BdTse} g(x,y)$$
and the proof is compete.
\end{proof}

As an application of Theorem \ref{thm1}, consider a valid inequality $$ax+cy\ge b$$ for $\UT$ where vectors $a$ and $c$ are of appropriate size.
If $ax+cy$ is arc-symmetric, that is, if 
\beqn a_{ij}^{vu}=a_{ji}^{vu}=a_{ij}^{uv}=a_{ji}^{uv} \text{ for all $(ij)\in A$ and $(uv)\in K$},
\label{eq:syma}\eeqn
 then 
$$\min_{(x,y)\in\UT} ax+cy  ~=~\min_{(x,y)\in\BdTs} \frac12ax+cy ~=~\hat b~\ge~b$$
and therefore
$$\frac12ax+cy\ge b$$ is   valid for $\BdTs$.
Conversely, if $ax+cy\ge b$ is   valid for $\BdTs$, then $2ax+cy\ge b$ is valid for $\UT$.
Therefore, valid inequalities that are defined by coefficients satisfying \eqref{eq:syma} can easily be translated between $\UT$ and $\BdTs$.

In particular,  consider an inequality  $cy\ge b$ that depends on the capacity variables  only.
Clearly $cy$ is  arc-symmetric and if $cy\ge b$  is a valid inequality for $\UT$, then it is  valid for $\BdTs$ as well.
Similarly, if $cy\ge b$ is valid for $\BdTs$, then it is also valid for $\UT$.
Combining this  with Corollary \ref{cor1} and equation \eqref{eq:BTse}, we make the following observation.

\begin{corollary}
Projections of the sets $\UT$, $\UTs$, $\UTse$,  $\BdTs$, and $\BdTse$  onto the space of capacity variables are equal.
\end{corollary}

Therefore, for the undirected case it suffices to study symmetric traffic matrices to characterize all valid inequalities that involve capacity variables only. The same, however, is not the case for the bidirected case.

\begin{remark}
The observation above suggests that the polyhedral structure of the bidirected network design problem is more complicated than that of the undirected one.
To demonstrate this, consider an instance of the bidirected network design problem that has a single facility type with $c_m=1$. 
Let $G=(N,E)$ be the complete (undirected) graph on three nodes and assume that there is no existing capacity.
It is known that \cite{BG:cnd} the projection of feasible solutions onto the space of capacity variables is 
$$\proj{y}{\BT} = \left\{ y\in\R^3\::\: y(i)\ge T(i) ~\forall i\in N,
~\sum_{e\in E} y_e\ge\max\Big\{\left\lceil\frac{\sum_{i\in N}T(i)}2\right\rceil,~\left\lceil\Theta\right\rceil\Big\}\right\},$$
where for $i\in N$ we let $j$ and $k$ denote the remaining two nodes in $N\sm\{i\}$ and define
$y(i)=y_{\{i,j\}}+y_{\{i,k\}}$, $T(i)=\left\lceil\max\{t_{ij}+t_{ik},t_{ji}+t_{ki}\}\right\rceil$, and
$\Theta=\max_{(i,j,k)\in\Pi}\{t_{ij}+t_{ik}+t_{jk}\}$ where $\Pi$ contains all triplets obtained by permuting the nodes in $N=\{1,2,3\}$.

Now consider $T^*$, the symmetric counterpart of $T$, and notice that  for all $i\in N$, we have $t^*_{ij}+t^*_{ik}=t^*_{ji}+t^*_{ki}$ and therefore $T^*(i)=\left\lceil t_{ij}+t_{ik}\right\rceil$.
In addition, for all $(i,j,k)\in\Pi$ we have $\Theta=t^*_{ij}+t^*_{ik}+t^*_{jk}$ and 
$\ceil{(T^*(1)+T^*(2)+T^*(3))/2}\ge \ceil\Theta$ as 
$$\left\lceil{
	\frac{
		\left\lceil t^*_{12}+t^*_{13}\right\rceil+
		\left\lceil t^*_{12}+t^*_{23}\right\rceil+
		\left\lceil t^*_{13}+t^*_{23}\right\rceil}
	2}\right\rceil
\ge \left\lceil \frac{2t^*_{12}+2t^*_{13}+2t^*_{23}}2\right\rceil.$$
Consequently, we have the simpler projection 
\beqa\proj{y}{\BTs} &=& \left\{ y\in\R^3\::\: y(i)\ge T^*(i) ~\forall i\in N,
			~\sum_{e\in E} y_e\ge\left\lceil\frac{\sum_{i\in N}T^*(i)}2\right\rceil\right\}\\
&=&\proj{y}{U(T/2)}\eeqa
for the undirected model.
\end{remark}

\section{Directed capacity model}


In the directed network design problem, the capacities are added on the arcs so that the total flow on an arc is limited by the capacity of the arc. The problem is modeled as:
\beqan \nonumber 
\text{min~~~~~} ~f(x,y)\hskip2cm&&\\[.3cm]
\text {s.t.~~~~~~~} 
\sum_{j \in \cNouti} \xijk ~~- \sum_{j \in \cNini} \xjik~ &=& d_i^k, ~~~ {\rm for\ } k \in \cK, \  i\in \cN, 
\label{eq:flowxxd} \\ \label{eq:capxxd}
\sum_{k \in \cK } \xijk  ~~  &\leq& \bar c_{ij}+ \sum_{m \in \cM } \ c_m  y_{m,ij},~~~ {\rm for\ } (ij) \in A,\\[.2cm] \label{eq:nonnegxxd} x,~y~&\ge& 0, \quad y\in\Z^{M\times A}.   
\eeqan
Let $\DT$ denote  the set of feasible solutions $(x,y)$ to inequalities \eqref{eq:flowxxd}--\eqref{eq:nonnegxxd}.
Let $y^1$ be the vector of capacity variables on arcs $(ij)$ with $i<j$ and
$y^2$ be the vector of capacity variables on arcs $(ij)$ with $i>j$.
Similarly define existing capacity vectors $\bar c^1, \ \bar c^2$.
Observe for a graph that has a reverse arc $(ji)$ for each arc $(ij)$ that adding constraints 
$y^1 = y^2$ to $\DT$ gives an equivalent formulation to $\BT$ provided that
$\bar c_{ij} = \bar c_{ji} = \bar c_e$ for all $(ij)$ and $e= \{i,j\}$. Let 
$D^=(T) = \DT \cap \{y \in \Z^{M \times A}: y^1 = y^2\}$. 
 
\begin{lemma} 
	Suppose $\bar c^1 = \bar c^2 = \bar c$.
The point
$(x,y)$ is feasible for $\BT$ if and only if $(x,y^1, y^2)$ with $y^1 = y^2 = y$ is feasible for $D^=(T)$. 
\end{lemma}
\begin{proof}
If $(x,y)$ is feasible for $\BT$, then letting $y^1 = y^2 = y$ gives a point $(x,y^1,y^2)$ in $\DT$ satisfying $y^1 = y^2$. Conversely, if $(x,y^1,y^2)$ is feasible for $D^=(T)$, then
$(x,y^1)$ is feasible for $\BT$ as 
\[
max \bigg \{ \sum_{k \in K} x_{ij}^k, \sum_{k \in K} x_{ji}^k \bigg \} \le \bar c_{\{i,j\}} + y^1_{\{i,j\}}.
\]
\end{proof}

Consequently, every valid inequality for $\DT$ yields a valid inequality for $\BT$ as shown in the next proposition.

\begin{proposition} \label{prop:di-to-bi}
If $\pi x + \beta^1 y^1 + \beta^2 y^2 \ge \pi_o$ is valid for $\DT$, then
$\pi x + (\beta^1 + \beta^2) y \ge \pi_o$ is valid for $\BT$.
\end{proposition}

\begin{proof}
Let $(x,y)$ be a feasible point in $\BT$. For $y^1 = y^2 = y$, we have 
\[
\pi x + (\beta^1 + \beta^2) y = \pi x + \beta^1 y + \beta^2 y = \pi x + \beta^1 y^1 + \beta^2 y^2 \ge \pi_o.
\]
Thus inequality holds.
\end{proof}

\begin{remark}
	Now we will show that some of the valid inequalities derived for B(T) in the literature can be obtained directly from 
	valid inequalities for D(T) via Proposition~\ref{prop:di-to-bi}.
	
	Consider a nonempty two-partition $(U,V)$ of the vertices of the network.
	Let $b^k$ denote the net demand of commodity $k$ in $V$ from $U$. 
	Let $A^+$ be the set of arcs directed from $U$ to $V$,
	$A^-$ be the set of arcs directed from $V$ to $U$, and $A = A^+ \cup A^-$.
	For $Q \subseteq K$ let $x_Q(S) = \sum_{k \in Q} x^k(S)$ and $b_Q = \sum_{k \in Q} b^k$.
	Without loss of generality, assume $b_Q \ge 0$.

	For $Q \subseteq K$ and $s \in M$ let $r_{s,Q} = b_Q' - \FLOOR{b_Q' / c_s}c_s$ and
	$\eta_{s,Q} = \CEIL{b_Q' / c_s}$, where  $b'_Q =  b_Q -\bar c(S^+) + \bar c(S^-)$. Then
	{\it multi-commodity multi-facility cut-set inequality} \citep{A:nd}
	\begin{equation}
	\label{eq:nd:lm-cutset}
	\sum_{m \in M} \phi_{s,Q}^+(c_m) y_m(S^+) \hspace{-0.5mm} + 
	\hspace{-0.5mm} x_Q(A^+ \hspace{-0.5mm} \setminus S^+)  \hspace{-0.5mm} + 
	\hspace{-1mm} \sum_{m \in M} \phi_{s,Q}^-(c_m) y_m(S^-) \hspace{-0.5mm} - \hspace{-1mm} x_Q(S^-) \hspace{-0.5mm} \geq 
	\hspace{-0.5mm} r_{s,Q} \eta_{s,Q},
	\end{equation}
	where
	\[
	\phi_{s,Q}^+(c) =
	\begin{cases}
	c - k(c_s-r_{s,Q}) & \text{if} \ kc_s \leq c < kc_s + r_{s,Q}, \\
	(k+1)r_{s,Q}       & \text{if} \ kc_s + r_{s,Q} \leq c < (k+1)c_s,
	\end{cases}
	\]
	and
	\[
	\phi_{s,Q}^-(c) =
	\begin{cases}
	c - kr_{s,Q}       & \text{if} \ kc_s \leq c < (k+1)c_s - r_{s,Q}, \\
	k(c_s-r_{s,Q})     & \text{if} \ kc_s - r_{s,Q} \leq c < kc_s, \\
	\end{cases}
	\]
	and $k \in \ZZ$, is valid for \DT.
	Above $\phi_{s,Q}^+$ and $\phi_{s,Q}^-$ are subadditive MIR functions written in closed form. 	Let $E^+$ be the undirected edges corresponding to $S^+$ and $E^-$ be the undirected edges corresponding to $S^-$. Then, Proposition~\ref{prop:di-to-bi} implies the bidirected valid inequality
		\begin{equation}
	\label{eq:nd:bi-lm-cutset}
	\sum_{m \in M} \phi_{s,Q}^+(c_m) y_m(E^+) \hspace{-0.5mm} + 
	\hspace{-0.5mm} x_Q(A^+ \hspace{-0.5mm} \setminus S^+)  \hspace{-0.5mm} + 
	\hspace{-1mm} \sum_{m \in M} \phi_{s,Q}^-(c_m) y_m(E^-) \hspace{-0.5mm} - \hspace{-1mm} x_Q(S^-) \hspace{-0.5mm} \geq 
	\hspace{-0.5mm} r_{s,Q} \eta_{s,Q}
	\end{equation}
	for \BT, which is given by \citet{RKOW:cutset}.

For the single facility case with $c_m = 1$, the directed inequality \eqref{eq:nd:lm-cutset}  reduces to 
\begin{align}
\label{eq:nd:flow-cut-set-inout}
r_Q y(S^+)+ x_Q(A^+ \setminus S^+) + (1-r_Q) y(S^-) - x_Q(S^-) \geq r_Q \eta_Q, 
\end{align}
where $r_Q = b_Q' - \FLOOR{b_Q'}$ and $\eta_Q = \CEIL{b_Q'}$. Proposition~\ref{prop:di-to-bi} implies the corresponding 
bidirected valid inequality
\begin{align}
\label{eq:nd:bi-flow-cut-set-inout}
r_Q y(E^+)+ x_Q(A^+ \setminus S^+) + (1-r_Q) y(E^-) - x_Q(S^-) \geq r_Q \eta_Q, 
\end{align}
for \BT.
\end{remark}

\ignore{
\citet{BG:cnd,CGS:nd-batch} generalize the basic cut-set inequalities \eqref{eq:nd:cut-set}
by incorporating the flow variables in addition to the capacity variables (see Figure~\ref{fig:cut-arcs}).	
For $S^+ \subseteq A^+$, $S^- \subseteq A^-$ and $Q \subseteq K$ consider the following relaxation of \FMS:
\begin{align*}
\bar c(S^+) + 	c y(S^+) +  x_Q(A^+ \setminus S^+)  - x_Q(S^-) & \ge   b_Q, \\
0 \leq \sum_{k \in Q} x^k_a & \leq  \bar c_a +   c y_a,  \ \forall a \in A, 
\end{align*}
which is written equivalently as
\begin{align*}
c \big [y(S^+)- y(S^-) \big] +  x_Q(A^+ \setminus S^+) + \big [ \bar c(S^-) + c y(S^-) - x_Q(S^-) \big] & \ge   
b_Q', \\
0 \leq \sum_{k \in Q} x^k_a & \leq  \bar c_a +  c y_a,  \ \forall a \in A. 
\end{align*}
where $b'_Q =  b_Q -\bar c(S^+) + \bar c(S^-)$.
Letting $r_Q = b_Q' - \FLOOR{b_Q'/c}c$ and $\eta_Q = \CEIL{b_Q'/c}$ and 
observing that  $x_Q(A^+ \setminus S^+)  \ge 0$, $\bar c(S^-) + c y(S^-) - x_Q(S^-) \ge 0$,
we can apply the MIR procedure to this relaxation to arrive at 
}

\section{Concluding Remarks}

We examined the relationship between three separately-studied capacity models for network design to provide a unifying approach and showed how to translate valid inequalities from one to the others. The directed model is the most general one and, therefore, valid inequalities for it can be translated into bidirected and undirected models. It is shown that the projections of the undirected and bidirected models onto the capacity variables are the same. We demonstrated that valid inequalities previously given for undirected and bidirected models can be seen as a consequence of the relationship between these models and the directed model.

\section*{Acknowledgement}

A. Atamt\"urk is supported, in part,
by grant FA9550-10-1-0168 from the Office of the Assistant Secretary of Defense for Research and Engineering.
O. G\"unl\"uk would like to thank the Numerical Analysis Group at Oxford Maths Institute for hosting him during this project. 

\bibliographystyle{plainnat}
\bibliography{../Paper/master}

\end{document}